\documentclass[12pt,reqno]{article}

\usepackage{amsmath,amsthm,amsfonts,amssymb,latexsym,mathrsfs,footmisc}

\newtheorem{theorem}{Theorem}[section]
\newtheorem{lemma}{Lemma}[section]
\newtheorem{corollary}{Corollary}[section]
\newtheorem{proposition}{Proposition}[section]
\newtheorem{conjecture}{Conjecture}[section]

\theoremstyle{definition}

\theoremstyle{remark}
\newtheorem{remark}{Remark}[section]

\numberwithin{equation}{section}
\newcommand{\exc}{{\rm exc\,}}

\title{Strong $q$-log-convexity of the Eulerian polynomials of Coxeter groups
\footnote{This work was supported in part by the National Natural Science
Foundation of China (Nos. 11201260, 11201191), the Key Project of
Chinese Ministry of Education (No. 212098), the Specialized
Research Fund for the Doctoral Program of Higher Education of China
(No. 20113705120002), and the National Science Foundation of
Shandong Province of China (No. ZR2011AL018).
\newline\hspace*{5mm}
 {\it Email addresses:} $^1$ liulily@mail.qfnu.edu.cn (L.L.
 Liu),  $^2$ bxzhu@jsnu.edu.cn (B.-X.
 Zhu)}}

\author{Lily Li Liu$^1$,\quad Bao-Xuan Zhu$^2$}
\date{\footnotesize $^1$ School of Mathematical Sciences, Qufu Normal University, Qufu 273165, PR China\\
$^2$ School of Mathematical Sciences, Jiangsu Normal University,
Xuzhou 221116, PR China}

\begin{document}
\maketitle

\begin{abstract}
In this paper we prove the strong $q$-log-convexity of the Eulerian
polynomials of Coxeter groups using their exponential generating
functions. Our proof is based on the theory of exponential Riordan
arrays and a criterion for determining the strong $q$-log-convexity
of polynomials sequences, whose generating functions can be given by
the continued fraction. As consequences, we get the strong
$q$-log-convexity of the Eulerian polynomials of types $A_n,B_n$,
their $q$-analogous and the generalized Eulerian polynomials
associated to the arithmetic progression
$\{a,a+d,a+2d,a+3d,\ldots\}$ in a unified manner.
\bigskip\\
{\sl MSC:}\quad 05A20; 05A15; 30F70
\bigskip\\
{\sl Keywords:}\quad Eulerian polynomials; Coxeter groups; Strong
$q$-log-convexity; Riordan array; Continued fraction
\end{abstract}

%%%%%%%%%%%%%%%%%%%%%%%%%%%%%%%%%%%%%%%%
\section{Introduction}
\hspace*{\parindent}
%%%%%%%%%%%%%%%%%%%%%%%%%%%%%%%%%%%%%%%%%%
The Eulerian polynomials $P(W,q)$, which enumerate the number of
descents of a (finite) Coxeter group $W$, is one of the classical
polynomials in combinatorics. During their long history, they arised
often in combinatorics and were extensively studied
(see~\cite{Bre94,Car59,Car78,Com74} and references therein). In
recent years, there has been a considerable amount of interesting
extensions and modifications devoted to these polynomials
(see~\cite{Bar11-1,Bar13,Cho,LW07-1,LW07-2,SV,WY05,Zhu} for
instance). In fact, Brenti showed that it is enough to study the
Eulerian polynomials for irreducible Coxeter groups~\cite{BB,Bre94}.
For Coxeter groups of type $A_n$, it is known that these polynomials
coincide with the classical Eulerian polynomials, whose properties
have been well studied from a combinatorial point of
view~\cite{Com74,FS,LW07-1,LW07-2,WY05}. Some properties of the
classical Eulerian polynomials can be generalized to the Eulerian
polynomials of type $B_n$, such as recurrence relations, the reality
of zeros, generating functions, unimodality and total positivity
properties~\cite{Bar11-1,Bre94,CG,SV}. In this paper, using their
exponential generating functions, we present the strong
$q$-log-convexity of many Eulerian polynomials of Coxeter groups,
which on one hand also generalizes the strong $q$-log-convexity of
the classical Eulerian polynomials~\cite{Zhu}, on the other hand
will give the strong $q$-log-convexity of types $A_n,B_n$, their
$q$-analogues and the generalized Eulerian polynomials associated to
the arithmetic progression $\{a,a+d,a+2d,a+3d,\ldots\}$ in a unified
manner.

Let $q$ be an indeterminate. For two real polynomials $f(q)$ and $g(q)$, denote $f(q)\geqslant_q
g(q)$ if the difference $f(q)-g(q)$ has only nonnegative
coefficients as a polynomial of $q$. We say that a real
polynomial sequence $\{f_n(q)\}_{n\geqslant0}$ is called {\it
$q$-log-convex} if
$$f_{n-1}(q)f_{n+1}(q)\geqslant_qf_n^2(q)$$ for $n\geqslant1$, and it is {\it strongly $q$-log-convex} if
$$f_{m-1}(q)f_{n+1}(q)\geqslant_qf_m(q)f_n(q)$$ for all $n\geqslant m\geqslant1$.
Clearly, the strong $q$-log-convexity of
polynomials sequences implies the $q$-log-convexity. However, the converse dose not follows.

As we know that many famous polynomials sequences, such as the Bell
polynomials~\cite{CWY,LW07-2}, the classical Eulerian
polynomials~\cite{LW07-2,Zhu}, the Narayana
polynomials~\cite{CWY2010-2}, the Narayana polynomials of type
$B$~\cite{CWY2010-1} and the Jacobi-Stirling
numbers~\cite{LZ,Zhu13}, are $q$-log-convex. Furthermore, almost all
of these polynomials sequences are strongly
$q$-log-convex~\cite{CWY,LZ,Zhu}. In this paper we give the strong
$q$-log-convexity of many Eulerian polynomials. Our
proof relies on the theory of exponential Riordan arrays and a
criterion of Zhu~\cite{Zhu} for determining the strong
$q$-log-convexity of polynomials sequences, whose generating
functions can be given by the continued fraction.

This paper is organized as follows. In section 2, using the theory
of exponential Riordan arrays and orthogonal polynomials, we first
give the continued fraction of the ordinary generating function of
the polynomials sequence, whose  exponential generating function
generalizes the exponential generating function of many Eulerian
polynomials. Then we obtain the strong $q$-log-convexity of the
polynomials sequence using the continued fraction and a criterion of
Zhu~\cite{Zhu}. As applications, we obtain the strong
$q$-log-convexity of the Eulerian polynomials of Coxeter groups,
including the Eulerian polynomials of types $A_n$, $B_n$, their
$q$-analogues defined by Foata and Sch\"utzenberger~\cite{FS} and
Brenti~\cite{Bre94} respectively, and the generalized Eulerian
polynomials associated to the arithmetic progression
$\{a,a+d,a+2d,a+3d,\ldots\}$~\cite{XTH} in a unified manner in
section 3. In section 4, we present some conjectures and open
problems. Finally, in the Appendix, we can obtain a quick
introduction to the exponential Riordan arrays and the orthogonal
polynomials used in this paper.

%%%%%%%%%%%%%%%%%%%%%%%%%%%%%%%%%%%%%%%%%%
\section{The strong $q$-log-convexity}
%%%%%%%%%%%%%%%%%%%%%%%%%%%%%%%%%%%%%%%%%%
%%%%%%%%%%%%%%%%%%%%%%%%%%%%%%%%%%%%%%%%%%
\hspace*{\parindent} In this section, we first give the continued
fraction of the ordinary generating function of the polynomials
sequence $\{T_n(q)\}_{n\geqslant0}$, whose exponential generating
function generalizes the exponential generating functions of many
Eulerian polynomials. Then using the continued fraction and a
criterion of Zhu~\cite{Zhu}, we prove the strong $q$-log-convexity
of the polynomials sequence $\{T_n(q)\}_{n\geqslant0}$.

\begin{theorem}\label{eo}
Suppose that the exponential generating function of the polynomials
sequence $\{T_n(q)\}_{n\geqslant0}$ has the following simple
expression
\begin{equation}\label{exponential}g(x)=\sum_{n\geqslant
0}\frac{T_n(q)}{n!}x^n=\left(\frac{(1-q)e^{a(1-q)x}}{1-q
e^{d(1-q)x}}\right)^b,\end{equation} for $a,b,d\in\mathrm{R}$. Then
the ordinary generating function of $\{T_n(q)\}_{n\geqslant0}$ can
be given by the continued fraction
\begin{equation}\label{ordinary}
h(x)=\sum_{n\geqslant
0}T_n(q)x^n=\frac{1}{1-s_0(q)x-\cfrac{t_1(q)x^2}{1-s_1(q)x-\cfrac{t_2(q)x^2}{1-s_2(q)x-\cfrac{t_3(q)x^2}{1-s_3(q)x-\cdots}}}},
\end{equation}
where
\begin{equation}\label{s-t}s_i(q)=(di+ab)+(di+bd-ab)q\ \  {\rm
and}\ \ t_{i+1}(q)=d^2(i+1)(i+b)q\end{equation} for $i\geqslant0.$
\end{theorem}

In order to prove this theorem, we need three lemmas. Using the
theory of the exponential Riordan arrays, the first lemma presents
that the production matrix $P$ of the exponential Riordan array
$L=[g(x),f(x)]$, where $g(x)$ is the exponential generating function
of $\{T_n(q)\}_{n\geqslant0}$ given by (\ref{exponential}), is
tri-diagonal.

\begin{lemma}\label{production}
The production matrix $P$ of the exponential Riordan array
$$L=[g(x),f(x)]=\left[\left(\frac{(1-q)e^{a(1-q)x}}{1-q
e^{d(1-q)x}}\right)^b,\frac{e^{d(1-q)x}-1}{d[1-q
e^{d(1-q)x}]}\right],$$  for $a,b,d\in\mathrm{R}$, is tri-diagonal.
\end{lemma}
\begin{proof}
In order to get the production matrix $P$, it suffices to calculate
$r(x)$ and $c(x)$. Recall that
$$r(x)=f'(\bar{f}(x)),\ \ \ c(x)=\frac{g'(\bar{f}(x))}{g(\bar{f}(x))},$$
where $\bar{f}(x)$ is the compositional inverse of $f(x)$.

By the direct calculation, we have
$$f'(x)=\frac{(1-q)^2e^{d(1-q)x}}{[1-q
e^{d(1-q)x}]^2}.$$ Note that the compositional inverse of $f(x)$
satisfies
$$f(\bar{f}(x))=\frac{e^{d(1-q)\bar{f}}-1}{d[1-q
e^{d(1-q)\bar{f}}]}=x.$$  Then we have
$$\bar{f}(x)=\frac{1}{d(1-q)}\ln\left(\frac{1+dx}{1+dq
x}\right).$$ Hence $$r(x)=f'(\bar{f}(x))=(1+dx)(1+dq
x)=1+d(1+q)x+d^2q x^2.$$

On the other hand, we have
$$g'(x)=b\left(\frac{(1-q)e^{a(1-q)x}}{1-q
e^{d(1-q)x}}\right)^{b-1}\frac{(1-q)^2e^{a(1-q)x}[a+(d-a)q
e^{d(1-q)x}]}{[1-q e^{d(1-q)x}]^2}.$$ So
\begin{eqnarray*}
c(x)&=&\frac{g'(\bar{f}(x))}{g(\bar{f}(x))}=\frac{b(1-q)[a+(d-a)q
e^{d(1-q)\bar{f}}]}{1-q e^{d(1-q)\bar{f}}}\\&=&ab(1+dq
x)+b(d-a)q(1+dx)\\&=&b[a+(d-a)q]+bd^2q x.
\end{eqnarray*}

Thus the production matrix $P$ of $L$ is tri-diagonal, where
\begin{equation}\label{production matrix}
P=\begin{pmatrix}s_0(q) & 1 & 0 & 0 & 0 & 0\cdots\\
t_1(q) & s_1(q) & 1 & 0 & 0 & 0\cdots\\0 & t_2(q) & s_2(q) & 1
& 0& 0\cdots\\0 & 0 & t_3(q) & s_3(q) & 1 & 0\cdots\\ 0 & 0 & 0 &
t_4(q) & s_4(q) & 1\cdots\\0 & 0 & 0 & 0 & t_5(q) &
s_5(q)\cdots\\\vdots & \vdots & \vdots & \vdots & \vdots & \vdots
\ddots \end{pmatrix},
\end{equation}
with $s_i(q)$ and $t_{i+1}(q)$ given by (\ref{s-t}).

\end{proof}

The second lemma constructs a family of orthogonal polynomials
related to the production matrix $P$ of the exponential Riordan
array $L=[g(x),f(x)]$.

\begin{lemma}\label{orthogonal}
Suppose that the production matrix $P$ of an exponential Riordan
array $L$ is tri-diagonal as above~(\ref{production matrix}). Then
we can construct a family of orthogonal polynomials $Q_n(x)$ defined
by
\begin{equation}\label{orth}Q_n(x)=[x-s_{n-1}(q)]Q_{n-1}(x)-t_{n-1}(q)Q_{n-2}(x),\end{equation}
with $Q_0(x)=1\  and\ Q_1(x)=x-s_0(q)$, where coefficients
$s_{n-1}(q)$ and $t_{n-1}(q)$ are given by the
expression~(\ref{s-t}).
\end{lemma}
\begin{proof}
In order to construct the family of orthogonal polynomials $Q_n(x)$,
it suffices to get the coefficient matrix $A$ such that
\begin{equation}\label{coefficient}
\begin{pmatrix}Q_0(x)\\ Q_1(x) \\ Q_2(x)\\ Q_3(x)\\\vdots\end{pmatrix}=A\begin{pmatrix}1\\x \\ x^2\\ x^3\\\vdots\end{pmatrix}.
\end{equation}
And by the condition and the Favard's Theorem~\ref{Bar-1} in
Appendix, we will get that the orthogonal polynomials $Q_n(x)$
satisfies the following
\begin{equation}\label{recurrence}
P\begin{pmatrix}Q_0(x)\\ Q_1(x) \\ Q_2(x)\\ Q_3(x)\\\vdots\end{pmatrix}=\begin{pmatrix}s_0(q) & 1 & 0 & 0 & 0 & 0\cdots\\
t_1(q) & s_1(q) & 1 & 0 & 0 & 0\cdots\\0 & t_2(q) & s_2(q) & 1
& 0& 0\cdots\\0 & 0 & t_3(q) & s_3(q) & 1 & 0\cdots\\\vdots & \vdots & \vdots & \vdots & \vdots & \vdots
\ddots \end{pmatrix}\begin{pmatrix}Q_0(x)\\ Q_1(x) \\ Q_2(x)\\ Q_3(x)\\\vdots\end{pmatrix}=\begin{pmatrix}xQ_0(x)\\ xQ_1(x) \\ xQ_2(x)\\ xQ_3(x)\\\vdots\end{pmatrix}.
\end{equation}
Then we have that the coefficient matrix $A$ satisfies
\begin{equation}\label{coefficient-1}
PA\begin{pmatrix}1\\ x \\ x^2\\ x^3\\\vdots\end{pmatrix}=A\begin{pmatrix}x\\x^2 \\ x^3\\ x^4\\\vdots\end{pmatrix}=A\bar{I}\begin{pmatrix}1\\ x \\ x^2\\ x^3\\\vdots\end{pmatrix},
\end{equation}
where $\bar{I}=(\delta_{i+1,j})_{i,j\geqslant0}$. Since the
polynomials sequence $\{x^k\}_{k\geqslant0}$ is linear independence.
So the coefficient matrices of the first and last polynomials
in~(\ref{coefficient-1}) are equal, i.e., $PA=A\bar{I}$. Since
$P=L^{-1}\bar{L}, \bar{I}=\bar{L}L^{-1}$. So we have that the
coefficient matrix $A$ will satisfy $L^{-1}\bar{L}A=A\bar{L}L^{-1}$.
Thus we can obtain that $L^{-1}$ is a coefficient matrix of the
orthogonal polynomials $Q_n(x)$. The proof of the lemma is complete.
\end{proof}
\begin{remark}From the proof of Lemma~\ref{orthogonal}, we have that the coefficient matrix of the orthogonal polynomials
$Q_n(t)$ is
\begin{eqnarray*}
L^{-1}=\left[\frac{1}{(1+d x)^{\frac{ab}{d}}(1+dq
x)^{\frac{bd-ab}{d}}},\frac{1}{d(1-q)}\ln\left(\frac{1+d x}{1+dq
x}\right)\right],\end{eqnarray*} which has been shown by
Barry~\cite{Bar13}. However our proof is more natural and based on
the algebraic method.
\end{remark}

The last lemma, obtained by Barry~\cite{Bar13}, gave the connection between the production matrix and the moment sequence of orthogonal polynomials.
\begin{lemma}[{\cite{Bar13}}]\label{Bar-lemma}
Let $L$, $T_n(q)$ and $Q_n(x)$ be as above. Then we have
$\{T_n(q)\}_{n\geqslant0}$ is the moment sequence of the associated
family of orthogonal polynomials $Q_n(x)$.
\end{lemma}

Now we can obtain that the ordinary generating function of
$\{T_n(q)\}_{n\geqslant0}$ is given by the continued
fraction~(\ref{ordinary}) from Theorem~\ref{Bar-2}, which proves
Theorem~\ref{exponential}.

Then we can present the strong $q$-log-convexity of
$\{T_n(q)\}_{n\geqslant 0}$ using the following criterion of
Zhu~\cite{Zhu}.
\begin{theorem}[{\cite[Proposition 3.13]{Zhu}}]\label{Sq-LCX-Zhu}
Given two sequences $\{s_i(q)\}_{i\geqslant 0}$ and
$\{t_{i+1}(q)\}_{i\geqslant 0}$ of polynomials with nonnegative
coefficients, let
\begin{equation*}
\sum_{n\geqslant
0}D_{n}(q)x^n=\frac{1}{1-s_0(q)x-\cfrac{t_1(q)x^2}{1-s_1(q)x-\cfrac{t_2(q)x^2}{1-s_2(q)x-\cfrac{t_3(q)x^2}{1-s_3(q)x-\cdots}}}}.
\end{equation*}
If $s_i(q)s_{i+1}(q)\geqslant_qt_{i+1}(q)$ for all $i\geqslant1$,
then the sequence $\{D_{n}(q)\}_{n\geqslant 0}$ is strongly
$q$-log-convex.
\end{theorem}

The main result of this section is the following.
\begin{theorem}\label{Sq-LCX}
The polynomials sequence $\{T_{n}(q)\}_{n\geqslant0}$ defined
by~(\ref{exponential}) forms a strongly $q$-log-convex sequence for
$b\geqslant0$ and $d\geqslant a\geqslant 0$
\end{theorem}
\begin{proof}
By Theorem~\ref{eo}, if the exponential generating function of
$\{T_{n}(q)\}_{n\geqslant0}$ has the expression~(\ref{exponential}),
then we have the ordinary generating function of
$\{T_{n}(q)\}_{n\geqslant0}$ can be given by the continued
fraction~(\ref{ordinary}). Note that $s_i(q)=(di+ab)+(di+bd-ab)q$
and $t_{i+1}(q)=d^2(i+1)(i+b)q$ for $i\geqslant0$ . So
\begin{eqnarray*}
&&s_i(q)s_{i+1}(q)-t_{i+1}(q)\\
&=&((di+ab)+(di+bd-ab)q)((di+d+ab)+(di+d+bd-ab)q)-d^2(i+1)(i+b)q\\
&=&(di+ab)(di+d+ab)+(di+bd-ab)(di+d+bd-ab)q^2\\
& &+((di+ab)(di+d+bd-ab)+(di+bd-ab)(di+d+ab)-d^2(i+1)(i+b))q\\
&=&(di+ab)(di+d+ab)+(di+bd-ab)(di+d+bd-ab)q^2\\
& &+((di+ab)(di+d+bd-ab)+abd(b-1)-a^2b^2)q\\
&\geqslant_q&(di+ab)(di+d+ab)+(di+bd-ab)(di+d+bd-ab)q^2+(ab^2d-a^2b^2)q\\
&\geqslant_q&0.
\end{eqnarray*}
The first and second inequalities hold by conditions $i,b\geqslant0$
and $d\geqslant a\geqslant 0$. Hence the polynomials sequence
$\{T_{n}(q)\}_{n\geqslant0}$ forms a strongly $q$-log-convex
sequence by Theorem~\ref{Sq-LCX-Zhu}
\end{proof}
%%%%%%%%%%%%%%%%%%%%%%%%%%%%%%%%%%%%%%%%%%
\section{the Eulerian polynomials of Coxeter groups}
%%%%%%%%%%%%%%%%%%%%%%%%%%%%%%%%%%%%%%%%%%
%%%%%%%%%%%%%%%%%%%%%%%%%%%%%%%%%%%%%%%%%%
\hspace*{\parindent} Given a finite Coxeter group $W$, define the
Eulerian polynomials of $W$ by $$P(W,q)=\sum_{\pi\in
W}q^{d_W(\pi)},$$ where $d_W(\pi)$ is the number of $W$-descents of
$\pi$. We refer the reader to Bj\"orner~\cite{BB} for relevant
definitions.

For Coxeter groups of type $A_n$, it is known that
$P(A_n,q)=A_n(q)/q$, the shifted Eulerian polynomials, whose strong
$q$-log-convexity was obtained by Zhu~\cite{Zhu}.
Since the exponential
generating function of $\{A_n(q)\}_{n\geqslant 0}$ and
$\{P(A_n,q)\}_{n\geqslant 0}$ is
\begin{equation}\label{A}\sum_{n\geqslant0}A_n(q)\frac{x^n}{n!}=\frac{(1-q)}{1-qe^{x(1-q)}}\end{equation}
and
\begin{equation}\sum_{n\geqslant0}P(A_n,q)\frac{x^n}{n!}=\frac{(1-q)e^{x(1-q)}}{1-qe^{x(1-q)}}\end{equation}
respectively (see~\cite[p. 244]{Com74}). So from Theorem~\ref{eo}, we have
\begin{equation*}
\sum_{n\geqslant
0}A_n(q)x^n=\frac{1}{1-x-\cfrac{qx^2}{1-(2+q)x-\cfrac{4qx^2}{1-(3+2q)x-\cfrac{9qx^2}{1-(4+3q)x-\cdots}}}},
\end{equation*}
with $s_i(q)=i+(i+1)q$ and $t_{i+1}(q)=(i+1)^2q$ for $i\geqslant0$ (see~\cite{Bar11-1} for instance). And
\begin{equation*}
\sum_{n\geqslant
0}P(A_n,q)x^n=\frac{1}{1-qx-\cfrac{qx^2}{1-(1+2q)x-\cfrac{4qx^2}{1-(2+3q)x-\cfrac{9qx^2}{1-(3+4q)x-\cdots}}}},
\end{equation*}
with $s_i(q)=(i+1)+iq$ and $t_{i+1}(q)=(i+1)^2q$ for $i\geqslant0$.

Obviously, we have the following result by Theorem~\ref{Sq-LCX}.
\begin{proposition}
The polynomials $P(A_n,q)$ and $A_n(q)$ form strongly $q$-log-convex sequences
respectively.
\end{proposition}

In~\cite{FS}, Foata and Sch\"utzenberger introduced a $q$-analog of
the classical Eulerian polynomials defined by
$$A_n(q;t):=\sum_{\pi\in S_n}q^{\exc(\pi)+1}t^{c(\pi)},$$
where $\exc(\pi)$ and $c(\pi)$ denote the numbers of excedances and
cycles in $\pi$ respectively. It is clear that $A_n(q;1)=A_n(q)$ is
precisely the classical Eulerian polynomial. Brenti showed
that the exponential generating function of
$\{A_n(q;t)\}_{n\geqslant0}$ is given by
\begin{equation}\label{q-A}\sum_{n\geqslant0}A_n(q;t)\frac{x^n}{n!}=\left(\frac{(1-q)e^{x(1-q)}}{1-qe^{x(1-q)}}\right)^t.\end{equation}
 So from Theorem~\ref{eo}, we have
\begin{equation*}
\sum_{n\geqslant
0}A_n(q;t)x^n=\frac{1}{1-tx-\cfrac{tqx^2}{1-(t+1+q)x-\cfrac{2(t+1)qx^2}{1-(t+2+2q)x-\cfrac{3(t+2)qx^2}{1-(t+3+3q)x-\cdots}}}}.
\end{equation*}
Here $s_i(q)=(t+i)+iq$ and $t_{i+1}(q)=(i+1)(t+i)q$ for $i\geqslant
0$.

Obviously, we have the following result by Theorem~\ref{Sq-LCX}.
\begin{proposition}
The polynomials $A_n(q;t)$ form a strongly $q$-log-convex sequence
for $t\geqslant 0$.
\end{proposition}

For Coxeter groups of type $B_n$, suppose that the Eulerian
polynomials of type $B_n$
\begin{equation*}P(B_n,q)=\sum_{k=0}^nB_{n,k}q^k,\end{equation*} where
$B_{n,k}$ is the Eulerian numbers of type $B_n$ counting the
elements of $B_n$ with $k$ $B$-descents. Then the Eulerian numbers
of type $B_n$ satisfy the recurrence
\begin{equation}\label{recurrenc-1}B_{n,k}=(2k+1)B_{n-1,k}+(2n-2k+1)B_{n-1,k-1}.\end{equation} Hence the
Eulerian polynomials of type $B_n$ satisfy the recurrence
\begin{equation}\label{recurrenc-2}P(B_n,q)=[(2n-1)q+1]P(B_{n-1},q)+2q(q-1)P'(B_{n-1},q).\end{equation} It is well
known that $P(B_n,q)$ have only real zeros (see~\cite{Bre94,SV} for
instance). Note that the exponential generating function of the
Eulerian polynomials of type $B_n$ has the following expression
\begin{equation}\label{recurrenc-3}\sum_{n\geqslant0}P(B_n,q)\frac{x^n}{n!}=\frac{(1-q)e^{x(1-q)}}{1-qe^{2x(1-q)}}\end{equation}
(see~\cite[Theorem3.4]{Bre94} and~\cite[Corollary3.9]{CG}). Hence
from Theorem~\ref{exponential}, we have the generating function of
the Eulerian polynomials of type $B_n$ is given by
\begin{equation*}
\sum_{n\geqslant
0}P(B_n,q)x^n=\frac{1}{1-(1+q)x-\cfrac{4qx^2}{1-3(1+q)x-\cfrac{16qx^2}{1-5(1+q)x-\cfrac{36qx^2}{1-7(1+q)x-\cdots}}}}.
\end{equation*}
Here $s_i(q)=(2i+1)(1+q)$ and $t_{i+1}(q)=4(i+1)^2q$ for $i\geqslant
0.$

Thus the strong $q$-log-convexity of $P(B_n,q)$ follows from
Theorem~\ref{Sq-LCX}.
\begin{proposition}
The polynomials $P(B_n,q)$ form a strongly
$q$-log-convex sequence.
\end{proposition}

From the definitions, if a sequence of polynomials is strongly $q$-log-convex, then it is $q$-log-convex. So we have the following corollary immediately.
\begin{corollary}\label{q-LCX}
The polynomials $P(B_n,q)$ form a
$q$-log-convex sequence.
\end{corollary}
\begin{remark}
From Liu and Wang~\cite[Theorem 4.1]{LW07-2}, we can also get
Corollary~\ref{q-LCX} using recurrences~(\ref{recurrenc-1})
and~(\ref{recurrenc-2}).
\end{remark}

Brenti~\cite{Bre94} defined a $q$-analogue of the polynomials $P(B_n,q)$ by
$$B_n(q;t):=\sum_{\pi\in B_n}q^{d_B(\pi)}t^{N(\pi)},$$ where $N(\pi):=|\{i\in[n],\pi(i)<0\}|.$
In particular, if $t=1$, then $B_n(q;1)=P(B_n,q)$, the Eulerian
polynomials of type $B_n$. And if $t=0$, then $B_n(q;0)=A_n(q)$, the
classical Eulerian polynomials. He showed that the exponential
generating function of $\{B_n(q;t)\}_{n\geqslant0}$ has the
following expression
\begin{equation}\label{q-B}
\sum_{n\geqslant0}B_n(q;t)\frac{x^n}{n!}=\frac{(1-q)e^{x(1-q)}}{1-qe^{x(1-q)(1+t)}}.
\end{equation}
 So from Theorem~\ref{eo}, the
generating function of $B_n(q;t)$ is given by
\begin{equation*}
\sum_{n\geqslant
0}B_n(q;t)x^n=\frac{1}{1-s_0(q)x-\cfrac{t_1(q)x^2}{1-s_1(q)x-\cfrac{t_2(q)x^2}{1-s_2(q)x-\cfrac{t_3(q)x^2}{1-s_3(q)x-\cdots}}}}.
\end{equation*}
Here $s_i(q)=(t+1)i+1+[(t+1)(i+1)-1]q$ and
$t_{i+1}(q)=[(t+1)(i+1)]^2q$ for $i\geqslant0.$

Thus the strong $q$-log-convexity of $B_n(q;t)$ follows from
Theorem~\ref{Sq-LCX}.
\begin{proposition}
The polynomials $B_n(q;t)$ form a strongly $q$-log-convex sequence
for $t\geqslant0$.
\end{proposition}

Recently, Xiong, Tsao and Hall~\cite{XTH} defined the general
Eulerian numbers $A_{n,k}(a,d)$ associated with an arithmetic
progression $\{a,a+d,a+2d,a+3d,\ldots\}$ as
\begin{equation*}
A_{n,k}(a,d)=(-a+(k+2)d)A_{n-1,k}(a,d)+(a+(n-k-1)d)A_{n-1,k-1}(a,d),
\end{equation*}
where $A_{0,-1}=1$ and $A_{n,k}=0$ for $k\geqslant n$ or
$k\leqslant-2.$ In particular, when $a=d=1$, $A_{n,k}(1,1)=A_{n,k}$,
the classical Eulerian numbers which enumerating the number of $A_n$
with $k-1$ descents. Similarly, the general Eulerian polynomials
associated  with an arithmetic progression
$\{a,a+d,a+2d,a+3d,\ldots\}$ can be defined as
\begin{equation*}
P_n(q,a,d)=\sum_{k=-1}^{n-1}A_{n,k}(a,d)q^{k+1}.
\end{equation*}
It is shown that the exponential generating function of
$\{P_n(q,a,d)\}_{n\geqslant 0}$ has the following expression
\begin{equation}\label{g-A}
\sum_{n\geqslant0}P_n(q,a,d)\frac{x^n}{n!}=\frac{(1-q)e^{ax(1-q)}}{1-qe^{dx(1-q)}}.
\end{equation}
So from Theorem~\ref{eo}, the generating function of
$\{P_n(q,a,d)\}_{n\geqslant0}$ is given by
\begin{equation*}
\sum_{n\geqslant
0}P_n(q,a,d)x^n=\frac{1}{1-s_0(q)x-\cfrac{t_1(q)x^2}{1-s_1(q)x-\cfrac{t_2(q)x^2}{1-s_2(q)x-\cfrac{t_3(q)x^2}{1-s_3(q)x-\cdots}}}},
\end{equation*}
with $s_i(q)=(di+a)+(di+d-a)q$ and $t_{i+1}(q)=(d(i+1))^2q$ for
$i\geqslant0$ (see~\cite{Bar13} for instance).

Thus the strong $q$-log-convexity of $P_n(q,a,d)$ follows from
Theorem~\ref{Sq-LCX}.
\begin{proposition}
The general Eulerian polynomials $P_n(q,a,d)$ associated with an
arithmetic progression $\{a,a+d,a+2d,a+3d,\ldots\}$ form a strongly
$q$-log-convex sequence for $d\geqslant a\geqslant 1$.
\end{proposition}

%%%%%%%%%%%%%%%%%%%%%%%%%%%%%%%%%%%%%%%%%%
\section{Concluding remarks and open problems}
%%%%%%%%%%%%%%%%%%%%%%%%%%%%%%%%%%%%%%%%%%
%%%%%%%%%%%%%%%%%%%%%%%%%%%%%%%%%%%%%%%%%%
\hspace*{\parindent} Let $a_0,a_1,a_2,\ldots$ be a sequence of
nonnegative numbers. The sequence is called {\it log-convex}
(respectively {\it log-concave}) if for $k\geqslant 1$,
$a_k^2\leqslant a_{k-1}a_{k+1}$ (respectively $a_k^2\geqslant
a_{k-1}a_{k+1}$). Let $\{a(n,k)\}_{0\leqslant k\leqslant n}$ be a
triangular array of nonnegative numbers. Define a linear
transformation of sequences by
\begin{equation}\label{trans}z_n=\sum_{k=0}^na(n,k)x_k,\ \ \ n=0,1,2,\ldots.\end{equation} We say
that the linear transformation~(\ref{trans}) preserve log-convexity
if it preserves the log-convexity of sequences, i.e., the
log-convexity of $\{x_n\}$ implies that of $\{z_n\}$. We also say
that corresponding triangle $\{a(n,k)\}_{0\leqslant k\leqslant n}$
preserve log-convexity.
%Chen {\it et. al.}\cite{CWY} presented that
%the Bessel transformation preserves the log-convexity.
 Liu and
Wang~\cite{LW07-2} obtained the binomial transformation, the
Stirling transformations of the first and second kind preserve
log-convexity respectively. They also proposed the following
conjecture, which is still open now.
\begin{conjecture}[{\cite{LW07-2}}]
The Eulerian transformation $z_n=\sum_{k=0}^nA_{n,k}x_k$ preserve
log-convexity.
\end{conjecture}

Similarly, we can raise the following problem related to the
Eulerian polynomials of type $B_n$.
\begin{conjecture}
Let
\begin{equation}\label{B}z_n=\sum_{k=0}^nB_{n,k}x_k\end{equation}
denote the Eulerian transformation of type $B_n$. Then the
transformation (\ref{B}) preserves log-convexity.
\end{conjecture}

%%%%%%%%%%%%%%%%%%%%%%%%%%%%%%%%%%%%%%%%%%
\section{Appendix}
%%%%%%%%%%%%%%%%%%%%%%%%%%%%%%%%%%%%%%%%%%
%%%%%%%%%%%%%%%%%%%%%%%%%%%%%%%%%%%%%%%%%%
\hspace*{\parindent} The {\it exponential Riordan
array}~\cite{Bar11-2,DFR,DS} denoted by $L=[g(x),f(x)]$, is an
infinite lower triangular matrix whose exponential generating
function of the $k$th column is $g(x)(xf(x))^k/k!$ for
$k=0,1,2,\ldots$, where $g(0)\neq0\neq f(0)$. An exponential Riordan
array $L=(l_{i,j})_{i,j\geqslant0}$ can also be characterized by two
sequences $\{c_n\}_{n\geqslant0}$ and $\{r_n\}_{n\geqslant0}$ such
that $$l_{0,0}=1,\ \ l_{i+1,0}=\sum_{j\geqslant0}j!c_jl_{i,j},\ \
l_{i+1,j}=\frac{1}{j!}\sum_{k\geqslant
j-1}k!(c_{k-j}+jr_{k-j+1})l_{i,j},$$ for $i,j\geqslant0$
(see~\cite{DFR} for instance). Call $\{c_n\}_{n\geqslant0}$ and
$\{r_n\}_{n\geqslant0}$ the $c-$ and $z-$ sequences of $L$
respectively. Associated to each exponential Riordan array
$L=[g(x),f(x)]$, there is a matrix $P=(p_{i,j})_{i,j\geqslant0}$,
called the {\it production matrix}, whose bivariate generating
function is given by $$e^{xy}[c(x)+r(x)y],$$ where
$$c(x)=\frac{g'(\bar{f}(x))}{g(\bar{f}(x))}:=\sum_{n\geqslant0}
c_nx^n,r(x)=f'(\bar{f}(x)):=\sum_{n\geqslant0} r_nx^n.$$ Deutsch
{\it et al.}~\cite{DFR} obtained the elements of production matrix
$P=(p_{i,j})_{i,j\geqslant0}$ satisfying
$$p_{i,j}=\frac{i!}{j!}(c_{i-j}+jr_{i-j+1}).$$
Assume that $c_{-1}=0.$ Note that $$P=L^{-1}\bar{L},
\bar{I}=\bar{L}L^{-1},$$ where $\bar{L}$ is obtained from $L$ with
the first row removed and
$\bar{I}=(\delta_{i+1,j})_{i,j\geqslant0}$, where $\delta_{i,j}$ is
the usual Kronecker symbol.

The following well-known results establish the
relationship among the orthogonal polynomials, three-term
recurrences, recurrence coefficients and the continued fraction of
the generating function of the moment sequence. The first result is
the well-known "Favard's Theorem".
\begin{theorem}[{\cite[Th\'eor\'eme~9 on
p. I-4]{Vie}}, or~{\cite[Theorem~50.1]{Wal}}]\label{Bar-1} Let
$\{p_n(x)\}_{n\geqslant0}$ be a sequence of monic polynomials with
degree $n=0,1,2,\ldots$ respectively. Then the sequence
$\{p_n(x)\}_{n\geqslant0}$ is (formally) orthogonal if and only if
there exist sequences $\{\alpha_n\}_{n\geqslant0}$ and
$\{\beta_n\}_{n\geqslant 1}$ with $\beta_n\neq0$ such that the
three-term recurrence
$$p_{n+1}(x)=(x-\alpha_n)p_n(x)-\beta_np_{n-1}(x)$$ holds, for
$n\geqslant1$, with initial conditions $p_0(x)=1$ and
$p_1(x)=x-\alpha_0.$
\end{theorem}
\begin{theorem}[{\cite[Propersition~1 (7) on
p. V-5]{Vie}}, or~{\cite[Theorem~51.1]{Wal}}]\label{Bar-2} Let
$\{p_n(x)\}_{n\geqslant0}$ be a sequence of monic polynomials, which
is orthogonal with respect to some linear functional $\mathcal{L}$.
For $n\geqslant1$, let
$$p_{n+1}(x)=(x-\alpha_n)p_n(x)-\beta_np_{n-1}(x),$$ be the corresponding three-term recurrence which is guaranted by
Favard's theorem. Then the generating function
$$h(x)=\sum_{k=0}^\infty\mu_kx^k$$ for the moments
$\mu_k=\mathcal{L}(x^k)$ satisfies
\begin{equation*}
h(x)=\frac{\mu_0}{1-\alpha_0x-\cfrac{\beta_1x^2}{1-\alpha_1x-\cfrac{\beta_2x^2}{1-\alpha_2x-\cfrac{\beta_3x^2}{1-\alpha_3x-\cdots}}}}.
\end{equation*}
\end{theorem}

%%%%%%%%%%%%%%%%%%%%%%%%%%%%%%%%%%%%%%%%%%%
%%%%%%%%%%%%%%% References
%%%%%%%%%%%%%%%%%%%%%%%%%%%%%%%%%%%%%%%%%%%
\bibliographystyle{amsplain}

\end{document}